\documentclass[10pt]{amsart}
\usepackage{amssymb,latexsym,amsthm}
\usepackage{amsmath}
\usepackage{amsfonts}
\usepackage{graphicx}
\usepackage{color}
\usepackage[numbers,comma,square,sort&compress]{natbib}
\usepackage{hyperref}

\newcommand{\gen}[1]{\langle{#1}\rangle^+}
\newcommand{\F}{\mathcal{F}}
\newcommand{\G}{\mathcal{G}}

\newcommand{\N}{\mathbb{N}}

\newcommand{\R}{\mathbb{R}}

\newcommand{\IFS}{{\rm{IFS}}}

\newcommand{\Pgen}[1]{\langle\langle{#1} \rangle\rangle^+}

\theoremstyle{plain}
\newtheorem{theorem}{Theorem}[section]
\newtheorem{maintheorem}{Theorem}

\newtheorem{lemma}[theorem]{Lemma}
\newtheorem{corollary}[theorem]{Corollary}

\newtheorem{proposition}[theorem]{Proposition}
\newtheorem{claim}[theorem]{Claim}
\newtheorem*{claim*}{Claim}

\theoremstyle{definition}
\newtheorem{remark}[theorem]{Remark}
\newtheorem{definition}[theorem]{Definition}

\title[{\tiny Minimal iterated function systems}]{Robust minimality of iterated function systems with two generators} 

\author[{\tiny{Ale Jan Homburg}}]{Ale Jan Homburg}
\address{KdV Institute for Mathematics, University of Amsterdam, Science park 904, 1098 XH Amsterdam, Netherlands}
\address{Department of Mathematics, VU University Amsterdam, De Boelelaan 1081, 1081 HV Amsterdam, Netherlands}
\email{a.j.homburg@uva.nl}

\author[{\tiny{Meysam Nassiri}}]{Meysam Nassiri}
\address{School of Mathematics, Institute for Research in Fundamental Sciences (IPM), 
 P. O. Box 19395-5746,  Tehran, Iran}
\email{nassiri@ipm.ir}
\thanks{M.N. was in part supported by a grant from IPM (No. 88370121).
}

\begin{document}

\begin{abstract}
We prove that any compact  manifold without boundary admits
a pair of diffeomorphisms that generates  $C^1$ robustly minimal
dynamics. 
We apply the results to the construction of blenders and robustly transitive skew product diffeomorphisms.
\end{abstract}

\maketitle

\setcounter{tocdepth}{1}
\tableofcontents

\section{Introduction}
In this paper we study robust minimality of iterated function systems.
By an iterated function system one means the action of the semigroup generated by a family of diffeomorphisms. The study of iterated function systems (IFS, for short), besides its own importance, has a remarkable role for understanding certain (single) dynamical systems. For instance, one can embed an IFS in to a skew-product over the full shift with sufficient number of symbols; and therefore in many dynamical systems exhibiting some form of hyperbolicity. 

Let $\mathcal{F}, \mathcal{G}$ be two families of diffeomorphisms on a compact manifold $M$. 
Denote 
\begin{align*} 
\mathcal{F}\circ\mathcal{G} &= \{ f\circ g ~|~ f\in\mathcal{F}, g\in\mathcal{G} \},        
       \end{align*}
and 
for $k\in\mathbb{N}$, 
\begin{align*} \mathcal{F}^{k} := \mathcal{F}^{k-1}\circ \mathcal{F}, ~~~~\mathcal{F}^0:=\{Id\}.
\end{align*}
Write $\gen{\mathcal{F}}$ for the semi-group generated by $\mathcal{F}$, that is, 
$\gen{\mathcal{F}}=  \bigcup _{n=0}^{\infty} \mathcal{F}^{k}$.
The action of the semi-group $\gen{\mathcal{F}}$ is called the \emph{iterated function system} (or IFS)  associated to $\mathcal{F}$ and we denote it by   $\rm{IFS}\,(\mathcal{F})$.
For $x\in M$, we write the orbits of the action of this semi-group  as
\begin{align*}
\gen{\mathcal{F}}(x) &= \{f(x): f\in \gen{\mathcal{F}}\}. 
\end{align*}

A sequence $\{x_n : n\in \mathbb{N} \}$ is called a branch of an orbit of $\rm{IFS}\,(\mathcal{F})$ if for each $n\in \mathbb{N}$ there is $f_n\in \mathcal{F}$ such that $x_{n+1}=f_n(x_n)$. We say that $\rm{IFS}\,(\mathcal{F})$ is minimal if any orbit has a branch which is dense in $M$. 
We say that a property P holds $C^{r}$ robustly for $\rm{IFS}\,(\mathcal{F})$ if it holds for $\rm{IFS}\,(\tilde{\mathcal{F}})$ for any family 
$\tilde{\mathcal{F}}$ whose elements are  $C^{r}$ perturbations of elements of $\mathcal{F}$.

Here, we prove the following 

\begin{maintheorem}\label{thm A}
Any boundaryless  compact manifold  admits a pair of diffeomorphisms that generates a $C^1$ robustly minimal iterated function system.
\end{maintheorem}

The number of generators in Theorem~\ref{thm A} is optimal.
Indeed, a single diffeomorphism can not be $C^1$
robustly minimal.  Recall that a diffeomorphism is minimal if any orbit is dense. 
Pugh's closing lemma yields the
existence of periodic points for $C^1$ generic diffeomorphisms. 
Moreover, not every manifold admits a minimal
diffeomorphism. In dimension two, the only manifold with minimal
diffeomorphisms is the torus.  Examples of robustly minimal iterated function systems with more generators have been constructed in \cite{ghs} and \cite{MR2765143}. 

The study of minimal IFS has important consequences to study and to construct  robustly transitive diffeomorphisms. Recall that a diffeomorphism is transitive if it has a dense orbit, and it is ($C^1$) robustly transitive if every nearby  diffeomorphism (in $C^1$ topology) is also transitive. 
Some parts of the proof of Theorem~\ref{thm A} are closely related 
to the construction of blenders. The notion of blender 
introduced by Bonatti and D\'{\i}az \cite{MR1381990} is the main tool to construct robustly transitive diffeomorphisms.
The results of this paper permit us 
to give a straightforward and general construction of 
blenders and robustly transitive dynamics in arbitrary dimension along the lines of \cite{np}. 
It can be summarized as follows (cf. Theorem \ref{thm C}): 

\begin{quote}
{\textit{Let $F$ be a diffeomorphism with invariant compact partially hyperbolic set  $\Gamma$ such that $F|_\Gamma$ is conjugate (appropriately) to an iterated function system $\IFS(\F)$.
If $\IFS(\F)$ is {\emph{strongly robustly minimal}}, then $F$ is {\emph{robustly transitive}}  on $\Gamma$.}}
\end{quote}

The strong robust minimality of an IFS  means roughly that one may change the perturbed family of maps at each iteration and still obtain dense orbit for all points (cf. Definition \ref{def strong}). The main technical in this direction is  Theorem \ref{thm strong}, which implies the following improvement of Theorem \ref{thm A}.
\begin{maintheorem}\label{thm B}
Any boundaryless  compact manifold  admits a pair of diffeomorphisms that generates a $C^1$ strongly robustly minimal  IFS.
\end{maintheorem}

We provide two constructions that prove Theorem~\ref{thm A}, in Sections~\ref{s:min}
and \ref{s:second}. Those constructions yield  different type of examples in the application to robustly transitivity.
In Section \ref{s:strong r m} we study the strong robust minimality of  IFS and then we complete the proof of Theorem \ref{thm B}. 
The connection to blenders, and an application to the construction of robustly
transitive skew product diffeomorphisms, is presented in Section~\ref{s:skew}.

\section{Minimal iterated function systems}\label{s:min}

In this section we prove Theorem~\ref{thm A}.

\subsection{Local robust minimality}\label{sec local}
\begin{definition}
 Let $\F$ be a family of maps on the metric space $X$, and $Y \subset X$.
We say that $\rm{IFS}\,(\mathcal{F})$ is \emph{minimal on $Y$} if  for any $x\in Y$, the orbit $\gen{\mathcal{F}}(x)$ has a branch which is dense in $Y$.  
\end{definition}

Recall that a map $\phi$ on a metric space $(X,d)$ is a contraction if and only if there is a constant 
$0<\kappa<1$ such that $d(\phi(x), \phi(y))  \leq  \kappa d(x,y)$, 
for all $x, y \in X$.

The following proposition is a modification of Propositions 2.3 of \cite{np} 
(see also \cite{thesis}). We give a short proof in Section \ref{s:strong r m}.

\begin{proposition} \label{pro ifs 1}
Let $U$ and $V $ be two open disks  in $\mathbb{R}^n$ containing $0$, and  
$\phi: U \to V$ be a diffeomorphism  with 
$\phi(0)=0$.  If $D \phi_0$ is a contraction, then there exists $k \in \mathbb{N} $ such that 
for any small $\varepsilon > 0$ there exist $\delta>0$ and 
vectors $c_1, \dots , c_k \in B_{\varepsilon}(0)$ such that
$$B_{\delta}(0)\subset \overline{\mathcal{O}_{\F}^+(0)},$$
where $\F=\{\phi+c_1, \dots , \phi+c_k \}$ 
and $\mathcal{O}_{\F}^+(0)$ is a branch of the orbit $\gen{\F}(0)$.
Moreover, $\IFS(\F)$ is robustly minimal on $B_{\delta}(0)$. 
\end{proposition}

\subsection{Proof of Theorem \ref{thm A}}\label{ss:proof}

\begin{definition}[weak hyperbolic point]
Let $p$ be a hyperbolic periodic point of $g$ of period $k$, we say that $p$ is
{\it $\delta$-weak hyperbolic} if 
$$ 1-\delta <  m(D_p g^k |_{E_p^s} ) < \parallel D_p g^k |_{E_p^s} \parallel <1 < m(D_p g^k |_{E_p^u}) < \parallel D_p g^k |_{E_p^u}  \parallel< \frac{1}{1-\delta}.$$
Here $E_p^s$ and $E_p^u$ are the stable and unstable subspaces at $p$.
\end{definition} 

\begin{proof}[Proof of Theorem~\ref{thm A}]
Let $M$ be a compact manifold of dimension $m\geq 2$. We refer 
to \cite{bdv} or \citep{1002.37017} for iterated functions systems in dimension one.
The proof has three steps.\\

{\bf Step 1.}
It is a standard consequence from Morse theory that one can 
take a Morse-Smale diffeomorphism $f_\circ$ on $M$ with a
unique attracting fixed point $o$.
We assume that the norm of $Df_\circ^{-1}(o)$ is sufficiently close 
to one (i.e. $o$ is a $\delta$-weak hyperbolic attractor with 
sufficiently small $\delta$).

We then modify
slightly the dynamics of $f_\circ$ in a neighborhood $U$ of $o$ in order to obtain a new diffeomorphism $f$ such that it has a periodic attracting
orbit $\{f^i(p)\}$ in $U$ with sufficiently large period $n$ and with very weak attraction. 
The period of the attracting orbit  and the rate of
its weak attraction depend to the dimension of the manifold. 

To do the modification we may assume that the open set $U$ is diffeomorphic to the product $\mathbb{D}^2\times \mathbb{D}^{m-2}$.  
If $m=2$, then we replace the dynamics on $U$ by the local diffeomorphism $f$ on $\mathbb{D}^2$ as in Figure~\ref{fig ifs}.

\begin{figure}[ht]
\begin{center}
\includegraphics[width=.45\textwidth]{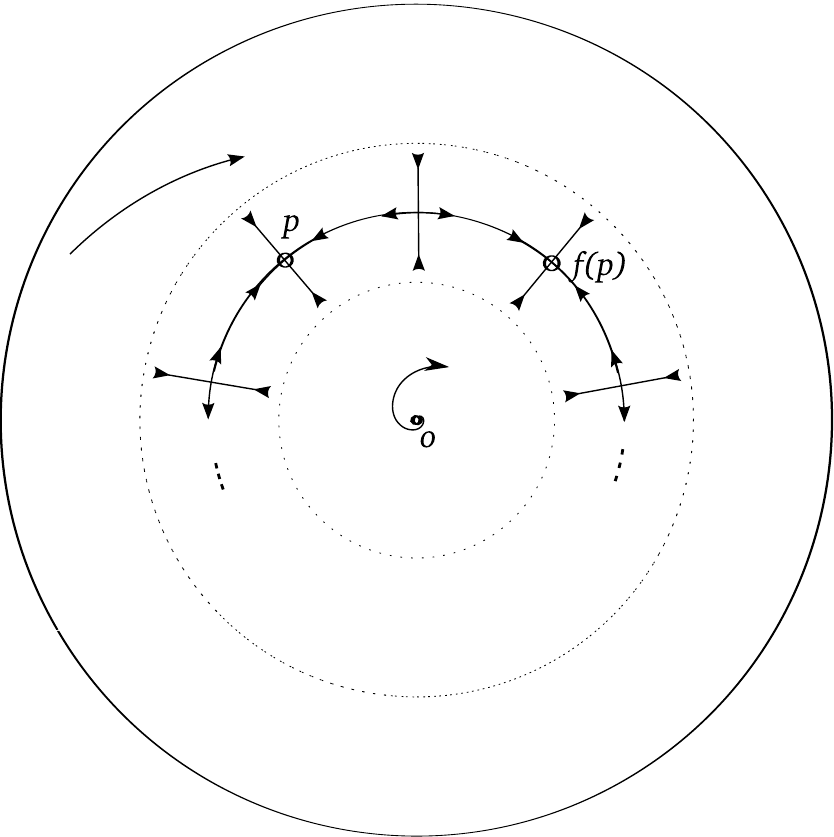}
\hspace{1cm}
\includegraphics[width=.45\textwidth]{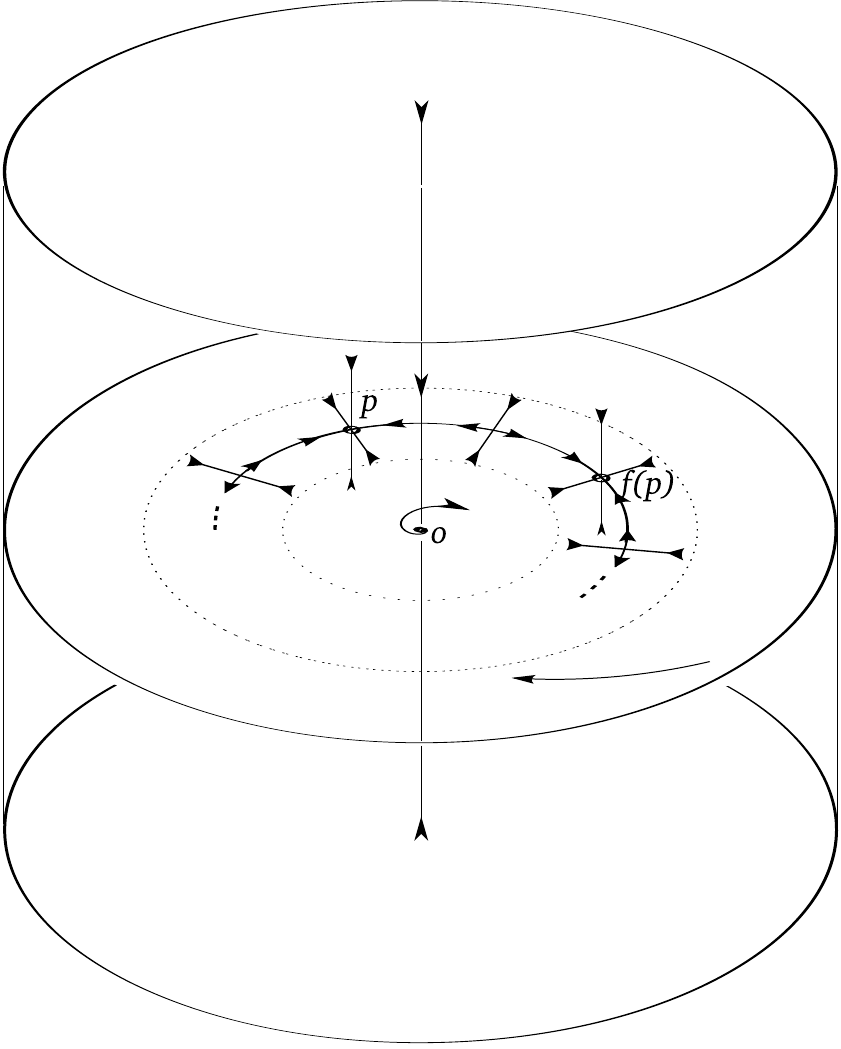}  
\caption{\label{fig ifs}Dynamics of $f$ if $\dim(M)=2$ (left picture) 
and $\dim(M) >2$ (right picture).}
\end{center}
\end{figure}

If $m>2$ then we define $h_1$ as the two dimensional map on $\mathbb{D}^2$ defined for the case $m=2$, 
and a contraction $h_2$ with a unique fixed point $o_2$ on $\mathbb{D}^{m-2}$ 
so that $o_2$ is a $\delta$-weak hyperbolic attractor, with a small $\delta>0$.  
Then we modify $f$ on $U$ to be the product $h_1\times h_2$.\\

{\bf Step 2.} (local minimality) 
We find a diffeomorphism $g$ arbitrarily close to $f$ in such a way that its dynamics 
close to the periodic orbit $\mathcal{O}_f(p) = \{f^i(p)\}$ is slightly different to $f$ 
and is chosen so that the IFS of $\{f,g\}$ is robustly minimal 
in a neighborhood of periodic orbit $\{f^{i}(p)\}_i$. 
Here we use the results of Section~\ref{sec local}.

Let $V_i:=f^i(V_0)$ be a sufficiently small neighborhoods of $f^i(p)$, for $i=0,\dots,$ $n-1$. 
We consider a local coordinate on  each $V_i$. 
By the assumptions,   $1-\epsilon<m(Df^n(p)) < || Df^n(p) || <1$. 
So, we may apply Proposition \ref{pro ifs 1} for the map $f^n$ on $V_0$, 
which gives very small vectors $c_1, \dots, c_k$, ($k<n$) 
on the local coordinate on $V_0$ so that the IFS of $\{f^n, T_{c_1}\circ f^n, \dots, T_{c_k}\circ f^n\}$ 
is robustly minimal on a neighborhood $D_0$ of $p$. 
Here,  $T_c$ is defined in local coordinates as the translation by the vector $c$.

Let $g_{_{\circ}}$ be a diffeomorphism on $M$ close to $f$ (so, it is Morse-Smale)  such that 
on $V_{n-i}$ it is equal to $f^{1-i}\circ T_{c_i} \circ f^{i}$.
Observe that $g_{_{\circ}}$ is well defined as $V_i$ are disjoint and $c_i$ are sufficiently small. 
Further, $f^{i-1}\circ g_{_{\circ}} \circ f^{n-i}$ on $V_0$ is equal to  $T_{c_i}\circ f^n$.
Let $g$ be any diffeomorphism $C^1$ close to $g_{_{\circ}}$. 
Then the IFS of $\{f, g\}$ is robustly minimal on $D_0$. 
We call such a local set with robustly minimal dynamics, a {\em blending region} of the iterated function system.

Observe that $\hat{g}:= g\circ f^{-2}$ is very close to $f^{-1}$ and so it is Morse-Smale. 
Moreover, the semigroup generated by $\{f, g \circ f^{-2}\}$ contains the semigroup generated by  $\{f, g\}$. 

From now on, we denote $g\circ f^{-2}$  by $\hat{g}$ for a given $g$.
Thus, for any $g$ close to $g_{_{\circ}}$, the ${\IFS}\,(\{f, \hat{g}\})$ is robustly minimal on $D_0$.
We may also assume that for any $g$ close to $g_{_{\circ}}$, $\hat{g}$ has a unique periodic repeller $p'$ such that $W^u(p')$ contains $D_0$.\\

{\bf Step 3.} (globalization)
We show that there exists a diffeomorphism $g$ close to $g_{_{\circ}}$ such that both positive and negative orbits 
of the blending region $D_0$ of $\IFS\,(\{f,\hat{g}\})$ cover the entire manifold $M$. 
If so, then for any point $x\in M$ and any open set $O\subset M$, there exist diffeomorphisms 
$h_1,h_2\in \gen{\{f, \hat{g}\}}$ such that, $h_1(x) \in D_0$ and $h_2^{-1}(O)\cap D_0 \neq \emptyset$.  
Since $D_0$ is a blending region of ${\IFS}\,(\{f, \hat{g}\})$, 
it follows that there exists $h_3\in \gen{\{f, \hat{g}\}}$ such that  $h_3(h_1(x)) \in h_2^{-1}(O)\cap D_0$, 
and so for $h:=h_2\circ h_3\circ h_1\in \gen{\{f, \hat{g}\}}$ we have $h(x)\in O$. 

Thus to  complete the proof of Theorem~\ref{thm A} it remains to show the above statement. 

Recall that for a Morse-Smale diffeomorphism the manifold $M$ is the union
of the stable manifolds of finitely many periodic points. 
Also, the union of basins of attraction of periodic attractors for a Morse-Smale diffeomorphism 
is an open and dense subset of the manifold: the stable manifold $W^s(p)$ for $f$ and the unstable manifold 
$W^u(p')$ for $\hat{g}$ are open and dense.

Pick $g$ so that $\hat{g}$ and $f$ have no common periodic points,
while $\hat{g}$ maps the periodic points of $f$ into the 
basin of attraction of $\{f^i(p)\}$. The negative orbits of $D_0$ then cover $M$. 
Similarly for positive orbits of $D_0$. They cover $M$ 
if $f^{-1}$ maps the periodic points of $\hat{g}$ into the basin of attraction
of the periodic attractor for $\hat{g}^{-1}$.  

Since these basins of attraction are open and dense in $M$,
these properties can be achieved for $g$  arbitrarily close to $g_0$.
This completes the proof of Theorem~\ref{thm A}.
\end{proof}

\section{Alternative construction}\label{s:second}

In this section we describe an alternative argument leading to Theorem~\ref{thm A}.
The local part of the argument, producing a blending region, 
is through an alternative construction.
The global part, the mechanism to iterate into and out of the blending region,
will be the same.
We start with the construction of a minimal iterated function systems
generated by affine maps on Euclidean spaces.

\subsection{Iterated function systems generated by affine maps}\label{s:rob}

We will provide an affine contraction $S$ and an affine expansion $T$ (i.e. $T^{-1}$ is a contraction) 
so that the iterated function system generated by $S$ and $T$ is minimal 
on all of $\mathbb{R}^m$.
Here we assume $m \ge 2$, compare \cite{bms} for minimal iterated functions systems 
on $[0,\infty)$ generated by affine maps.
Moreover, $S \circ T$ will be a contraction and the iterated function system generated by the two affine
contractions $S$ and $S\circ T$ possesses an attractor with nonempty interior.
Similar ideas were independently used by Volk \cite{vol13} in his construction of persistent attractors for endomorphisms.

Consider the rotation $R$ in $\mathbb{R}^m$, $m \ge 2$,
\begin{align*}
 R (x_1 , \ldots , x_m) &= (\pm x_m , x_1 , \ldots , x_{m-1}),  
\end{align*}
where the sign is such that $R \in SO(m)$, i.e., a minus sign for even $m$ 
and a plus sign for odd $m$.
Define further the translation $H(x_1,\ldots,x_m) = (x_1 + s , x_2 , \ldots,x_m)$ and the affine map $S$ on ${\mathbb R}^m$ by
\begin{align}\label{e:S}
 S (x_1, \ldots, x_m) &= H \circ r R (x_1,\ldots,x_m) = ( \pm r x_m + s , r x_1, \ldots, r x_{m-1}),
\end{align}
for constants $0<r<1$, $s>0$.
Likewise, define an affine map $T$ on  ${\mathbb R}^m$ by
\begin{align}\label{e:T}
 T (x_1, \ldots, x_m) &= (-a x_1, a x_2 , \ldots, a x_{m-1} , -a x_m - 2\frac{s}{r})
\end{align}
with $a>1$.
Similar to $S$, $T$ is the composition of a map from $SO(m)$ which is multiplied by a factor, $a$, 
and a translation.
Note that $S$ is a contraction by $r<1$, while $T$ is an expansion as $a>1$.  
Compute
\begin{align*}
 S \circ T (x_1, \ldots, x_m) &= H^{-1} \circ a r R (x_1,\ldots,x_m) 
\\
&=  (\mp a r x_m - s , - a r x_1 , a r x_2, \ldots , a r x_{m-1}).
\end{align*}
The affine map $S \circ T$ is a contraction for $a r <1$.

\begin{lemma}\label{l:lemme}
There are constants $0<r<1$, $s>0$, $a>1$ with $a r <1$
so that the iterated function system $\rm{IFS}\,(\mathcal{G})$ with
$\mathcal{G} =\{ S, S \circ T\}$
has an attractor $\triangle$ with nonempty interior.
Moreover, this interior contains the fixed point of $T$.
\end{lemma}

\begin{proof}
Define the box $B(1,v_2,\ldots,v_m)$ with corners $(\pm 1,\pm v_2, \ldots , \pm v_m)$.
We will find $r,s,a$ and $v_2,\ldots,v_m$ so that 
\begin{align}\label{e:expB}
 S (B) \cup S \circ T (B) &\supset B
\end{align}
This is a consequence of the following conditions,
\begin{align}\label{e:B}
 r v_m +s > 1, ~~ 0> -r v_m + s, ~~ r > v_2, ~~ r v_2 > v_3, ~~ \ldots , 
 ~~ r v_{m-1} > v_m.
\end{align}
In fact,  \eqref{e:B} implies
\begin{align}\label{e:aB}
 a r v_m +s > 1,  ~~ 0> - a r v_m + s,  ~~ a r > v_2,  ~~  a r v_2 > v_3, ~~  \ldots, ~~  a r v_{m-1} > v_m,
\end{align}
and \eqref{e:B} and \eqref{e:aB} together give \eqref{e:expB}.
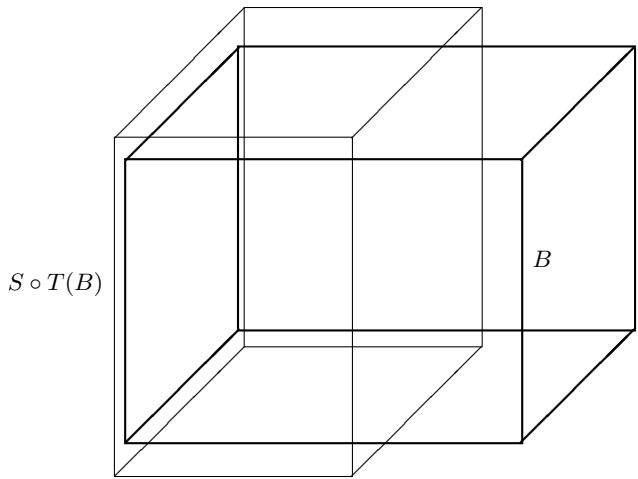
\begin{figure}[htbp]
\begin{center}

\setlength{\unitlength}{3108sp}%
\begingroup\makeatletter\ifx\SetFigFont\undefined%
\gdef\SetFigFont#1#2#3#4#5{%
  \reset@font\fontsize{#1}{#2pt}%
  \fontfamily{#3}\fontseries{#4}\fontshape{#5}%
  \selectfont}%
\fi\endgroup%
\begin{picture}(5043,3759)(1291,-5293)
\thinlines
{\color[rgb]{0,0,0}\put(2161,-5281){\line( 1, 1){1035}}
}%
{\color[rgb]{0,0,0}\put(4051,-5281){\line( 1, 1){1035}}
}%
{\color[rgb]{0,0,0}\put(2161,-5281){\framebox(1890,2700){}}
}%
{\color[rgb]{0,0,0}\put(2161,-2581){\line( 1, 1){1035}}
}%
{\color[rgb]{0,0,0}\put(4051,-2581){\line( 1, 1){1035}}
}%
\thicklines
{\color[rgb]{0,0,0}\put(5401,-5011){\line( 1, 1){900}}
}%
{\color[rgb]{0,0,0}\put(3151,-4111){\framebox(3150,2250){}}
}%
{\color[rgb]{0,0,0}\put(2251,-5011){\line( 1, 1){900}}
}%
{\color[rgb]{0,0,0}\put(2251,-2761){\line( 1, 1){900}}
}%
{\color[rgb]{0,0,0}\put(5401,-2761){\line( 1, 1){900}}
}%
\thinlines
{\color[rgb]{0,0,0}\put(3196,-4246){\framebox(1890,2700){}}
}%
\thicklines
{\color[rgb]{0,0,0}\put(2251,-5011){\framebox(3150,2250){}}
}%
\put(5491,-3616){\makebox(0,0)[lb]{\smash{{\SetFigFont{9}{10.8}{\familydefault}{\mddefault}{\updefault}{\color[rgb]{0,0,0}$B$}%
}}}}
\put(1306,-3796){\makebox(0,0)[lb]{\smash{{\SetFigFont{9}{10.8}{\familydefault}{\mddefault}{\updefault}{\color[rgb]{0,0,0}$S\circ T(B)$}%
}}}}
\end{picture}%
 
\caption{The image $S\circ T (B)$ covers over half the box $B$. The image $S(T)$ covers the 
remaining part of $B$, so that $S(B) \cup S\circ T(B)$ contains $B$. 
\label{f:box}}
\end{center}
\end{figure}

The repelling fixed point of $T$ is located at $(0,\ldots,0, \frac{-2s}{r(a+1)})$.
It lies in $B$ if 
\begin{align}\label{e:fpinB}
 2 s &< v_m r (a+1).
\end{align}
Now, \eqref{e:B} and \eqref{e:fpinB} can be satisfied by
taking suitable $v_m < \ldots < v_2 < r$ all near $1$
and $s = 1 - v_m^2$ near 0 so that \eqref{e:B} holds, 
and $a$ with $a r < 1$.

Since $S$ and $S\circ T$ are contractions, there is a ball $O$ 
that is mapped into itself by both $S$ and $S \circ T$, 
i.e., $\mathcal{G} (O) \subset O$.
Thus
\begin{align*}
\triangle &= \lim_{p\rightarrow\infty}\mathcal{G}^{p}(O)
\end{align*}
is a nonempty compact set that is invariant for $\rm{IFS}\,(\mathcal{G})$.
Since $S$  and $S \circ T$ are contractions,
$\triangle$ is the unique compact set that is invariant for
$\mathcal{L}$ \citep{MR625600}. 
Because $\mathcal{L} (B) \supset B$, the set $\triangle$ contains $B$.
\end{proof}

\begin{corollary}
The iterated function system $\rm{IFS}\,(\mathcal{G})$ is minimal on
$\mathbb{R}^m$.
\end{corollary}

The proof of Lemma~\ref{l:lemme} gives more than its statement 
as it includes arguments for $C^1$ robust occurrence of invariant sets with nonempty interior.
Denote by $\mathrm{Diff}^1(\mathbb{R}^m)$ the set of diffeomorphisms on $\mathbb{R}^m$, endowed
with the $C^1$ compact-open topology.

\begin{corollary}\label{c:rob}
There exists a neighborhood $W\subset\mathrm{Diff}^1(\mathbb{R}^m)\times\mathrm{Diff}^1(\mathbb{R}^m)$ of 
$(S,S\circ T)$ such that for each $(f_1,f_2)$ in this neighborhood,
 $\rm{IFS}\,(\mathcal{F})$ with $\mathcal{F}=\{f_1,f_2\}$ 
admits an invariant set with non-empty interior.
\end{corollary}

In rescaled coordinates $(y_1,\ldots,y_m) = h (x_1,\ldots,x_m) = (\delta x_1,\ldots,\delta x_m)$, $\delta >0$, the attractor 
$\triangle$ is multiplied by a factor $\delta$.
The affine maps computed in the $(y_1,\ldots,y_m)$
coordinates become \[h \circ S \circ h^{-1} (y_1,\ldots,y_m) = (\pm r y_m + \delta s , r y_1,\ldots,r y_{m-1})\] 
and  
\[h \circ T \circ h^{-1} (y_1,\ldots,y_m) = (-a y_1 , a y_2,\ldots,a y_{m-1}, - a y_m - 2 \frac{\delta s}{r});\]
the maps are unaltered except for the translation vector which is multiplied by $\delta$.
In other words, if $S$ and $T$ are affine maps as above 
so that $\mathcal{G} (\mathbb{R}^m; S,T)$ has an attractor $\triangle$, than replacing $s$ in their expressions 
by $\delta s$ yields an iterated function system with attractor $\delta \triangle$.

\subsection{Second proof of Theorem~\ref{thm A}}\label{s:proof}

The proof of Theorem~\ref{thm A} in Section~\ref{ss:proof} consists of a local part,
the construction of a blending region in steps 1 and 2, and a global part in step 3 involving
a mechanism to move from the blending region to other parts of the manifold, and back, by iterating.
The same construction can be followed here, where Lemma~\ref{l:lemme} used in a chart on the manifold
provides a blending region. Compare \cite{ghs}.

Specifically we obtain two diffeomorphisms $g_1,g_2$ on $M$ 
generating a robustly minimal iterated function system with the following properties:
\begin{enumerate}
 \item for $i=1,2$, $g_i$ has a unique attracting fixed point  $p_i$ and a unique repelling fixed point $q_i$,
 \item there is a blending region, containing $p_1$ and $q_2$, on which $g_1$ and $g_1\circ g_2$ 
are contractions.
\end{enumerate}
The diffeomorphisms $g_1,g_2$ are not $C^1$ close to the identity, 
but the construction can be done so that $m(Dg_1), |Dg_1|, m(Dg_2), |Dg_2|$ are everywhere close to $1$.

\section{Strong robust minimality}\label{s:strong r m}
In this section we introduce the notion of strong robust minimality for iterated function systems.

\subsection{Notations and definitions}\label{s:def robust}
Let $X$ be a metric space, and $\F$ be a family of maps on $X$.
\begin{itemize}
\item For any $Y\subset X$ and $x\in X$, we denote $\F(Y):= \bigcup_{\phi\in\F} \phi(Y)$ and $\F(x):= \F(\{x\})$. 
\item If each element of  $\F$ is invertible, then we denote $\F^{-1}= \{ f^{-1}~|~ f\in\F\}$.
\item Let $\bf{F}=(\F_1, \F_2, \dots)$, where  $\F_i$ is a family of maps on $X$, for any $n\in \N$. then we denote $\F^{(n)}= \F_n\circ \dots \circ\F_1$, $\F^{(1)}=\F_1$, $\F^{(0)}= \{Id\}$, 
$$ \Pgen{ {\bf{F}} }:=  \bigcup_{n=0}^{\infty}  \F^{(n)},$$
 and its orbits by $\Pgen{{\bf{F}}}(x) := \{f(x): f\in \Pgen{ {\bf{F}} } \}$ for some $x\in X$.
\item We also denote $\F_{m}^{(n)}:=\F_{m+n}  \circ \dots  \circ \F_{m+1}$ for $n,m \in \N\cup\{0\}$.
\item Let $x\in X$. A sequence $\{x_n : n\in \mathbb{N} \} \subset X$ is called a branch of the orbit  $\Pgen{{\bf{F}}}(x)$ if for any $n\in \mathbb{N}$ there is $f_n\in \F^{(n)}$ such that $x_{n+1}=f_n(x_n)$.
\end{itemize}

Let  $\F$ be a family (with $k$ elements) of $C^r$ diffeomorphisms of a manifold $M$.
A neighborhood $\mathcal{U}$ of $\F$ is the set of all families $\F'$ whose elements are $C^{r}$ perturbations of elements of $\mathcal{F}$.

\begin{definition}\label{def strong}
We say that $\rm{IFS}\,(\mathcal{F})$ is $C^{r}$ \emph{strongly} robustly minimal if there is  a neighborhood $\mathcal{U}$ of $\F$ such that for any $x\in M$ and 
any sequence $\bf{F}=(\F_1, \F_2, \dots)$ in $\mathcal{U}$ 
there is a branch of the orbit of $\Pgen{{\bf{F}}}(x)$ which is dense in $M$.
\end{definition}

It is clear that this is stronger than robust minimality defined in the introduction.
We will see how this notion is natural and effective when one applies the robust minimality of IFS to construct examples of robustly transitive diffeomorphisms.

We say that a map $\phi$ on a metric space $(X,d)$ is  bi-Lipschitz  with constants $\lambda>0$ and $\kappa$, if  $\lambda d(x,y) \leq d(\phi(x), \phi(y))  \leq  \kappa d(x,y)$ for all $x, y \in X$.  
If $0<\kappa<1$ then $\phi$ is a contraction. Then, we also say  $\lambda$ is the  
\textit{contraction-lower-bound} of $\phi$.

\subsection{Sufficient conditions for strong robustness}\label{s:strong thms}
\begin{theorem} \label{thm strong}
 Let $M$ be a boundaryless compact manifold  and $D$ be an open subset of $M$.
 Let $\F$ be a finite family of diffeomorphisms of $M$ such that 
\begin{enumerate}
\item for any $f\in \F$, $f|_{\overline{D}}$ is a contraction,
\item $\overline{D} \subset \F(D)$  
\end{enumerate}
Then $\IFS(\F)$ is strongly robustly minimal on $D$. 

Let $\G$ be a family of diffeomorphisms  of $M$ such that $\F \subset \G$ and $M = \G(D) = \G^{-1}(D)$. Then $\IFS(\G)$ is strongly robustly minimal (on $M$).
\end{theorem}

Before proving the theorem we prove a basic lemma.

\begin{lemma}\label{lem sufficient}
Let $X$ be a connected metric space,  
and  $D\subset X$ be a bounded open set. Let  $0<\lambda<\kappa <1$. 
Let $\bf{F}=(\F_1, \F_2, \dots)$ be such that for any $n\in\N$, 
$\F_n$ is a family of homeomorphisms of $X$ such that 
\begin{enumerate}
\item for any $f\in \F_n$, $f|_D$ is bi-Lipschitz  with constants $\lambda$ and $\kappa$, 
\item $\overline{D} \subset \F_n(D)$, 
\item  $Z_n$ is  $\delta$-dense in $D$, 
\end{enumerate}
where $Z_n = \{x\in D ~|~ f(x)=x {\rm~for~some~} f\in\F_n\}$, and  $2 \delta >0$ is smaller than the Lebesgue number of a covering $\mathcal{A}$ of $\overline{D}$ with $\mathcal{A}\subseteq \{f(D)\cap \overline{D} ~|~ f\in \F_n\}$. 
Then for any $x\in D$, $\Pgen{{\bf{F}}}(x)$ has a branch which is dense in $D$.
\end{lemma}

\begin{proof}
It is enough to prove the following. 
\begin{claim}\label{af}
Let $B\subset D$ be an open ball. Then there is $n\in \N$ such that for  any $m\in \N\cup\{0\}$ there is  $f\in  \F_m^{(n)}= \F_{m+n}  \circ \dots  \circ \F_{m+1}$ such that $f(D)\in B$.
\end{claim}

\begin{proof}
Assume that $B=B_r(x)$ for some $x\in D$. 
Let $k\in \N$ be  (the smallest integer) larger than $\ln(\delta/{\rm diam}(D)) / \ln(\kappa)$.

(a) If $r \geq 2 \delta$, $B'=B_{\delta}(x)$ be the ball of radius $\delta$ and same center.
Then, it follows from (3) that for any $i\in \N$,  $B'$ contains some $z\in Z_i$. Thus, there is some $f_i\in \F_i$ with a contacting fixed point in $B'$. 
By applying this fact at most $k$ times we see that there is $f\in \F_{m+k} \circ \dots  \circ \F_{n+2} \circ \F_{m+1}$ such that $f(D)\subset B$. In particular, the lemma follows for $n=k$.

(b) If $r<2\delta$, let $k(r)\in \N$ be (the smallest integer) larger than $\ln(r/(2\delta))/\ln(\lambda)$. 
Then, it follows from (2) that for any $i\in\N$, there is some $g_i\in \F_i$ such that $B\subset g_i(D)$. Then by (1),  $g_i^{-1}(B)\subset D$ and it contains a ball of radius $r \lambda^{-1}$.   
By applying this fact at most $k(r)$ times we see that  there is $g\in \F_{i+k(r)} \circ \dots  \circ \F_{i+1}$ such that $g^{-1}(B\cap g(D))$ contains a ball of radius $2\delta$. Thus from (a), for any $m\in N\cup \{0\}$, there is $h\in \F_{m+k} \circ \dots  \circ \F_{m+2} \circ \F_{m+1}$ such that $h(D)\subset g^{-1}(B\cap g(D))$. 
In particular, $(g\circ h)(D) \subset B$. For $i=m+k$,  $n=k+k(r)$ and $f=g\circ h$,  it follows that  $f(D) \subset B$ and  $f \in  \F_{m+n} \circ \dots  \circ \F_{m+2} \circ \F_{m+1}$, completing the proof of claim.
\end{proof}
Now, by considering a countable topological base for $D$ and applying the claim repeatedly one obtains a branch of $\Pgen{{\bf{F}}}(x)$ which is dense in $D$.
\end{proof}

\begin{remark}
The condition (2) and (3) in Lemma \ref{lem sufficient} are equivalent to what are called 
the covering and the well-distributed properties (respectively) in \cite[Definition 2.4]{np}.
\end{remark}

\begin{proof}[Proof of Theorem \ref{thm strong}]
Let $\delta>0$ be a number such that $2\delta$ is smaller that  the Lebesgue number of the covering $\mathcal{A}= \{f(D)\cap \overline{D} ~|~ f\in \F\}$ of $\overline{D}$.

Let $Y_n$ be the set of  fixed points of elements of $\F^n$.
It follows from the main result of \cite{MR625600} that there is $n_0\in \N$ large enough such that for any $n\geq n_0$, $Y_n$ is $(\delta/2)$-dense in $D$. 

It is easy to see that the family $\F\cup \F^{n_0}$ satisfies the hypothesis (1)-(3) of Lemma \ref{lem sufficient}  for some bi-Lipschitz constants $0<\lambda^{1/2}<\kappa^2<1$.
Let $\mathcal{V}$ be a sufficiently small  neighborhood of $\F$, and let $\bf{F}=(\F_1, \F_2, \dots)$ be a sequence of perturbations of $\F$ in $\mathcal{V}$.
Then, for any $m\in \N$, $\F_m^{(n_0)}$ is close to $\F^{n_0}$, and the set of its fixed points (denoted by $Z_m^{(n_0)}$) is $\delta$-dense in $D$.
Consequently, for any $n, m\in \N$, the family $\F_n\cup \F_m^{(n_0)}$ satisfies the hypothesis  (1)-(3) of Lemma \ref{lem sufficient}  for  bi-Lipschitz constants $\lambda$ and $\kappa$.

Now, let $\{B_i\}$ be a topological base of $D$, so that each $B_i$ is a ball.
Fix $i\in \N$, and let   $B=B_i$ be a ball of radius $r$.
By the proof of previous Claim (to apply part (a) we consider the families $\F_i^{(n_0)}$, and to apply part (b) we consider the families $\F_i$), there are non-negative integers $k$ and $k(r)$ such that for any $m\in \N$, 
there is $f\in \F_m^{(n)}$ such that $f(D)\subset B$, where $n=kn_0+ k(r)$. 

Using this statement inductively, one find a dense branch in $D$ for any orbit $\Pgen{{\bf{F}}}(x)$, $x\in D$.
\end{proof}

The following gives a short proof for Proposition \ref{pro ifs 1}.
\begin{proposition}
There exist a family of maps $\F$ as in Proposition \ref{pro ifs 1} such that $\IFS(\F)$ is strongly robustly minimal on some open ball $B_\delta(0)$.
\end{proposition}

\begin{proof}
Let $\lambda>0$ such that  $B_{\lambda}(0) \subset D\phi_0 (B_1(0))$.
Then, we consider a cover of the closed unit ball of $\R^n$ by $k$ balls of radius $\lambda$, i.e.,
$$\overline{B_1(0)}\subset \bigcup_{i=1}^k B_{\lambda}(b_{i}) = \bigcup_{i=1}^k (B_{\lambda}(0) + b_{i}).$$ 
It follows that for $\delta>0$ small enough, 
$$\overline{B_\delta(0)}\subset \bigcup_{i=1}^k (\phi(B_{\delta}(0)) + \delta b_{i}).$$ 
To complete the proof, it is enough to apply the previous theorem for $D=B_{\delta}(0)$ and the family $\F=\{\phi + c_1, \dots, \phi + c_{k}\}$, where $c_i=\delta b_{i}$.
\end{proof}

\subsection{End of proof of Theorem~\ref{thm B}}
\begin{proof}[Proof of Theorem~\ref{thm B}]
It is easy to see that the pair of diffeomorphism (in both proofs) of Theorem \ref{thm A} satisfies the hypothesis of Theorem \ref{thm strong} and so it generates a strongly robustly minimal IFS.
\end{proof}


\section{Skew products and robust transitivity}\label{s:skew}

Following ideas developed in \cite{np}, we discuss relations between 
robustly minimal iterated function systems constructed in the previous sections
and blenders.
We demonstrate how, through these relations, the minimal iterated
function systems provide elementary constructions of  
robustly transitive skew products.
More precisely, we show how 
the constructions in this paper
provide an alternative construction for
one of the main results from 
\cite{MR1381990} on robustly transitive skew products
(see Theorem~\ref{t:B} below).

The material in the previous sections is readily translated 
to skew product systems of diffeomorphisms over shift maps; 
a blending region will give rise to a symbolic blender.
Blenders are certain hyperbolic invariant sets 
used in the construction of robustly transitive dynamics,
as explored in \cite{MR1381990}
(see also \citep{bdv,bondia12}).
The original notion of blenders is in a context of partial hyperbolicity with
one dimensional central directions.
The generalization of blenders to higher dimensional central directions
was first considered in \citep{thesis,np}, where symbolic blenders
and their geometric models are introduced and applied to 
robustly transitive dynamics in symplectic settings.
A further study of their geometric and dynamical properties 
is in \cite{barkirai12}. 

\subsection{Symbolic blenders}

Let $g_1,g_2$ be two diffeomorphisms on $M$ generating a robustly 
minimal iterated function system
as in the previous sections.
For reasons of definiteness and clarity we assume the 
properties listed in Section~\ref{s:proof}.
Write $h_1 = g_1$, $h_2 = g_1 \circ g_2$ and let $\mathcal{H} = \{ h_1 , h_2 \}$. 
Recall that $\rm{IFS}\,(\mathcal{H})$ acts minimally on a set $\Delta$.
More specific, there are open sets $E_{in} \subset \Delta \subset E_{out}$ on which
\begin{equation}\label{e:ifsinclusion}
E_{in} \subset \mathcal{H}(E_{in}) \subset \Delta \subset \mathcal{H}(E_{out}) \subset E_{out}
\end{equation}
and $h_1$ and $h_2$ are contractions on $E_{out}$.

With $\Sigma = \{1,2\}^\mathbb{Z}$, let $H_0$ on $\Sigma \times M$ be defined by
\begin{align}\label{e:H+}
H_0( \omega , y) = (\sigma \omega , h_{\omega_0} (y)).
\end{align}
where $(\sigma \omega)_k = \omega_{k+1}$ is the left shift operator.
For the one-sided symbol space $\Sigma_+ = \{ 1,2\}^\mathbb{N}$,
a map $H_{0,+}$ on $\Sigma_+ \times M$ is defined likewise, using the same formula. 
The symbol spaces $\Sigma_+$ and $\Sigma$ will be endowed with the product topology,
and we write $\Sigma = \Sigma_- \times \Sigma_+$.

A natural situation is where $\Sigma$ occurs, 
through a topological conjugacy, 
as an invariant set of a diffeomorphism.
This gives sense to notions of (partial) hyperbolicity and of 
invariant manifolds such as stable manifolds.
It will in fact be the context of our applications on robustly 
transitive skew product systems.  
Henceforth, we work in the setup of diffeomorphisms 
$f:N \to N$ on a compact manifold
$N$ possessing a maximal hyperbolic invariant set in an open $U \subset N$ on which
$f$ is topologically conjugate to $\sigma$ on $\Sigma$.
We consider skew product systems $(x,y) \mapsto (f(x) , g_y(x))$ on $N \times M$;
write $\mathcal{S}^1 (N,M)$ for the set of such diffeomorphisms endowed with the 
$C^1$ topology.
Note that $C^1$ small
perturbations of $f$ yield a perturbed hyperbolic set in $U$
on which the dynamics remains
topologically conjugate to $\sigma$ acting on $\Sigma$.
For convenience, 
as we restrict diffeomorphisms to their hyperbolic sets,
we write $\mathcal{S}^1(\Sigma, M)$ for the $C^1$ skew product maps
$H : \Sigma \times M \to \Sigma \times M$,
\begin{align}
 H(\omega , y) = (\sigma \omega , h_\omega (y)),
\end{align}
endowed with the 
$C^1$ topology in the sense just described.

Suppose $H$ is partially hyperbolic on $\Sigma \times M$ in the 
sense that contraction and expansion rates of $\sigma$ dominate those
of $y \mapsto h_{\omega} (y)$.
That is, with $E^s \oplus E^u$ denoting the splitting in 
stable and unstable directions for $\sigma$,  
\begin{align*}
 |D\sigma(\omega)|_{E^s_\omega}| < m (Dh_\omega(y)) \le |Dh_\omega (y)| 
< m(D\sigma(\omega)|_{E^u_\omega}) 
\end{align*}
for each $y\in M$ and $\omega \in \Sigma$.
The partially hyperbolic skew product systems in $\mathcal{S}^1(\Sigma,M)$ form
an open subset of $\mathcal{S}^1(\Sigma,M)$; we may assume
that $H_0$ is partially hyperbolic.
We refer to \citep{bdv} for more on partial hyperbolicity.
As a consequence of partial hyperbolicity, 
there are strong stable and strong unstable manifolds of points in $\Sigma \times M$
that project to the stable and unstable manifolds of $\sigma$ under 
the projection to the base space $\Sigma$.
These strong stable manifolds provide a strong stable 
lamination $\mathcal{F}^{ss}$ 
of $\Sigma \times M$.
For definiteness, a local strong stable manifold is a compact part of
a strong stable manifold that projects to some
 $\Sigma_- \times \{ \omega_+\}$ under 
the projection to the base space $\Sigma$.
Likewise, strong unstable manifolds of points in $\Sigma \times M$ 
provide a strong unstable 
lamination $\mathcal{F}^{uu}$ of $\Sigma \times M$.
Local strong unstable manifolds project to $\{\omega-\} \times \Sigma_+$ 
under the projection to the base space $\Sigma$.
A strong stable or unstable lamination of a set 
is called minimal if each of its leaves lies dense in the set.

The maximal invariant set $B$ of $H$ in $\Sigma \times E_{out}$ 
in the following proposition is called a symbolic blender \cite{np}.
The local unstable set of $B$ is to be interpreted as
the union of local strong unstable manifolds through points in $B$.

\begin{proposition}\label{p:blender}
Any skew product map $H \in \mathcal{S}^1 (\Sigma,M)$
sufficiently close to $H_0 (\omega,m) = (\sigma\omega , h_{\omega_0} (m))$
possesses a maximal  invariant set $B \subset \Sigma \times E_{out}$ with the following properties:
\begin{itemize}
 \item $H$ is topologically mixing on $B$,
 \item the strong unstable lamination of $B$ is minimal, 
 \item any local strong stable manifold inside $\Sigma \times E_{in}$ 
        intersects the local unstable set of $B$. 
\end{itemize}
\end{proposition}

\begin{proof}
The image $H_{0,+}^{k} (\Sigma_+ \times E_{in})$ contains $2^k$ strips 
$\Sigma_+ \times U_i \subset \Sigma_+ \times E_{out}$ 
with diameter of $U_i$ going to $0$ as $k \to \infty$ 
and together covering $\Sigma_+ \times E_{in}$. 
Take an open set $U\subset \Sigma_+ \times \Delta$. A high iterate $H_{0,+}^n (U)$ 
contains a strip $\Sigma_+ \times J$  
in $\Sigma_+ \times \Delta$. 
By the above description of iterates $H_{0,+}^{k} (\Sigma_+ \times E_{in})$,
we see that further iterates $H_{0,+}^{n+k} (U)$ 
contain strips in $\Sigma_+ \times \Delta$,
that lie increasingly dense in it as $k \to \infty$. 

This reasoning also applies to small perturbations $H_+$ 
of $H_{0,+}$, where also
the fiber maps may depend on all of $\omega$ instead of just
$\omega_0$ (for one dimensional central directions this is also pursued in \cite{ajh11}).
Suppose $H_+(\omega,x) = (\sigma \omega , h_\omega (x))$ is such that
$h_\omega$ depends continuously on $\omega$ and is uniformly close to $h_{\omega_0}$. 
We note the following changes in the reasoning.
The inclusions \eqref{e:ifsinclusion} get replaced by
\begin{equation}\label{e:tildeH+}
\Sigma \times E_{in} \subset H_+ (\Sigma_+ \times E_{in}), \qquad
  H_+(\Sigma_+ \times E_{out}) \subset \Sigma_+ \times E_{out}
\end{equation}
An iterate $H^n_+$ maps $\Sigma_+ \times E_{out}$ to $2^n$ strips $R^i$.
The diameter of a strip is the maximal real number $r$ so that each 
$R_i \cap \{\omega \} \times E_{out}$
is contained in a ball of radius $r$. 
The map $H_+$ on $\Sigma_+ \times E_{out}$ acts by contractions 
in the fibers $\{\omega\}\times E_{out}$.
Hence $H_+^n(\Sigma_+ \times E_{in})$ are $2^n$ strips
with diameter going to zero as $n\to \infty$ and together
covering $\Sigma_+ \times E_{in}$.
Iterates under $H_+$ of  $\Sigma_+ \times E_{in}$ or
$\Sigma_+ \times E_{out}$ therefore converge to an 
invariant set that contains $\Sigma_+ \times E_{in}$.
As above, for any open set $U$,
$H^{n+k}_+ (U)$ contain strips
lying increasingly dense in the invariant set of $H_+$ as $k$ increases.

We proceed with skew products over the shift operator on two sided symbol spaces. 
For $H$ $C^1$-close to $H_0$, there is a strong stable foliation close to the affine foliation, 
i.e. with leaves close to $\Sigma_- \times \{ \omega_+\} \times \{m\}$. 
Write $\pi^+: \Sigma \times E_{out} \to \Sigma_+ \times E_{out}$ for the continuous
projection along local strong stable manifolds.
The existence of a strong stable foliation that is close to the affine foliation
means that $\pi^+$ is $C^0$-close to the identity.

We can copy the previous reasoning.
Observe that $H^n$ maps a curve $\{\omega_-\} \times \Sigma_+ \times \{m\}$
to $2^n$ curves that are each  a graph of a 
map $\Sigma_+ \to \Sigma_- \times E_{out}$.
Likewise  $H^n$ maps $\Sigma \times E_{out}$ to $2^n$ strips $R^i$.
Since $\pi^+$ is close to the identity, \eqref{e:tildeH+} gets replaced by
\begin{equation*}
\Sigma \times E_{in} \subset \pi^+ H (\Sigma_+ \times E_{in}), \qquad
  \pi^+ H (\Sigma_+ \times E_{out}) \subset \Sigma_+ \times E_{out}.
\end{equation*}
Invariance of the strong stable foliation gives that these inclusions also hold for
iterates of $H$.
Further, by uniform $C^1$-closeness of $h_\omega$ to $h_{\omega_0}$:
\begin{equation*}
 \lim_{i\to \infty}
\text{diameter}(R_i) = 0. 
\end{equation*}
The diameter of a strip, as before, is the maximal real number $r$ so that each 
$R_i \cap \{\omega\} \times E_{out}$
is contained in a ball of radius $r$. 
This proves the proposition. 
\end{proof}

\subsection{Robustly transitive skew product diffeomorphisms}

As an application of symbolic blenders from Proposition~\ref{p:blender} 
we show how it gives constructions of robustly transitive diffeomorphisms,
obtaining a main result in \cite{MR1381990} in a straightforward manner.

Let $N,M$ be compact manifolds and let $f: N\mapsto N$
be a diffeomorphism with a compact hyperbolic 
locally maximal invariant set 
$\Lambda_f$,
on which $f$ is topologically mixing.
Below we will use the fact that unstable manifolds for $f$ of points 
in $\Lambda_f$ lie dense in $\Lambda_f$, see e.g.
\citep[Section~18.3]{kathas97}.
Recall that a $C^1$ small perturbation 
$\tilde{f}$ of $f$ possesses a hyperbolic attractor $\Lambda_{\tilde{f}}$
near $\Lambda_f$. Moreover, $\tilde{f}$ restricted to $\Lambda_{\tilde{f}}$
is topologically conjugate to $f$ restricted to $\Lambda_f$.
Let $R_1, \dots, R_k$ be a Markov partition for $\Lambda_f$, through which, $f|_{\Lambda_f}$ is conjugate by $\varphi$ to the full shift $\sigma:\Sigma \to \Sigma$ with $k$ symbols.

\begin{theorem}[\cite{MR1381990}]\label{t:B}
There is a diffeomorphism $F: N\times M \to N\times M$, $F (x,y) = (f(x),g(x,y))$
with $f$ as above, 
that is topologically mixing on $\Lambda_f \times M$.
Moreover, $F$ is robustly topologically mixing in $\mathcal{S}^1(N\times M)$; i.e.
there is an open neighborhood $U$ of $F$ in $\mathcal{S}^1(N\times M)$ so that 
each $\tilde{F}(x,y) = (\tilde{f}(x),\tilde{g}(x,y))$ from $U$
is topologically mixing on $\Lambda_{\tilde{f}} \times M$.
\end{theorem}

\begin{proof}
Let $\Omega_f \subset \Lambda_f$ be a Smale horseshoe for an 
iterate $f^k$ of $f$, i.e. an invariant set for $f^k$ on which $f^k$ 
is topologically conjugate to the shift $\sigma$ on the symbol space 
$\Sigma = \{ 1,2 \}^\mathbb{Z}$ (see e.g. \cite[Theorem~6.5.5]{kathas97}).
For simplicity we assume $k=1$ for now.
With the Smale horseshoe comes a Markov partition of two sets $U_1,U_2$ 
covering $\Omega_f$, for which $\Omega_f$ is the maximal invariant set 
in $U_1\cup U_2$.  

Consider a skew product system $(x,y) \mapsto (f (x) , g(x,y))$
with the following properties:
\begin{enumerate}
 \item through the conjugation of $f|_{\Omega_f}$ with $\sigma$, 
the skew product system on $\Omega_f \times M$ is given by 
\begin{align} 
 G( \omega , y) &= (\sigma \omega , g_{\omega_0} (y)),
\end{align}
where $g_1,g_2$ are as in Section~\ref{s:proof}, 
 \item for $(x,y) \in U_i\times M$,  $i=1,2$, $g(x,y)$ does not depend on $x$.
\end{enumerate}
Such a skew product system exists as $g_1$ and $g_2$ are isotopic to the identity.

We now consider the diffeomorphism $F$ restricted to $\Omega_f \times M$. 
The repelling fixed point $q_2$ of $g_2$ in $\Delta$
gives a fixed point $(2^\infty , q_2)$ for $G$
that is repelling in $\{ 2^\infty\} \times M$.
This corresponds to a fixed point $Q_2$ for $F$ that is repelling 
within its central fiber. 
Similarly $F$ has a fixed point $P_1$ coming from the fixed point 
$(1^\infty, p_1)$ for $G$,
this fixed point is contracting within its central fiber.
Note that $W^u(Q_2)$ is dense in $\Lambda_f \times M$,
since unstable manifolds for $f$ are dense in $\Lambda_f$.  
We claim that 
\begin{align*}
 W^u (Q_2) &\subset \overline{W^{uu}(P_1)},
\end{align*}
establishing that $W^{uu}(P_1)$ lies dense in $\Lambda_f \times M$.
Namely, take a point $q \in W^u(Q_2)$ and a neighborhood $V$ of it,
iterate backwards and note that $H^{-m} (V)$ intersects $W^{uu}(P_1)$
by the existence of a symbolic blender of $F$ on $\Omega_f \times M$
(compare \citep[Lemma~1.9]{MR1381990}).
Density of $W^{uu}(P_1)$ implies that $H$ is topologically mixing 
on $\Lambda_f \times M$,
since iterates of an open set intersect the local stable manifold of $P_1$ 
and thus accumulate onto $W^{uu}(P_1)$.
This construction is robust under perturbations of $F$, 
see also Proposition~\ref{p:blender},
proving that $F$ is robust topologically mixing.
\end{proof}


\subsection{Sufficient conditions for robust transitivity}
Let $f$ be as in the previous section.
First we state a simple application of the well known results of \cite{hps}.

\begin{lemma}[Normal hyperbolicity]\label{lem hps}
Let $F: N\times M \to N\times M$ be a skew-product diffeomorphism such that $F(x,y) = (f(x),g(x,y))$
with $f$ as above, and such that $g(x,\cdot)$ is uniformly dominated by $f|_{\Lambda_f}$
Then, for any $G$ in a $C^1$ neighborhood of $F$ in ${\rm{Diff}}(M\times N)$ the following holds:   
\begin{itemize}
\item There is a unique continuation $\Gamma_{G}$ of $\Gamma_F:=\Lambda_f\times M$,
\item $G|_{\Gamma_{G}}$ is conjugate by $\phi_G$ to the $C^1$ skew product $H_G$ on $ \Sigma\times M$:
$$(\omega, y) \mapsto (\sigma(\omega),h_{G,\omega}(y) )$$
\item for any $\omega \in \Sigma$, $\phi_G^{-1}(\{\omega\}\times M)$ is diffeomorphic to $M$ and it is $C^1$ close to 
$\phi_F^{-1}(\{\omega\}\times M)= \{\varphi^{-1}(\omega)\}\times M$.
\end{itemize}
\end{lemma}

\begin{definition}
If $F$ satisfies the previous lemma, then we say $F$ is {\emph{robustly transitive on $\Gamma_F$}} if for any $G$ some neighborhood of $F$, $\Gamma_{G}$ is a transitive set for $G$. The analogues notion will be used for topologically mixing instead of transitivity.
\end{definition}

On the other hand, one can proof easily the following (see \cite{np} for a proof).

\begin{lemma}
Let $F: \Sigma\times M \to \Sigma\times M$, $F (\omega,y) = (\sigma(\omega),g_{\omega_0}(y))$, where 
$\G=\{g_i\}$ be a family of diffeomorphisms such that $\IFS(\G)$ is minimal.
Then $F$ is transitive; and the strong-unstable lamination is minimal.
\end{lemma}

As a corollary one obtains the following criterion for robust transitivity:

\begin{theorem}[Sufficient condition for robust transitivity]\label{thm C}  
Let $F$ as in Lemma \ref{lem hps} and suppose that for $i=1, \dots, k$,  
$$F|_{R_i \times M} = f|_{R_i}\times g_i$$
and such that $\IFS (\{g_1, \dots, g_k\})$ is strongly robustly minimal.
Then, $F$ is robustly transitive on $\Gamma_F$, and the  strong-unstable lamination is robustly minimal in $\Gamma_F$.
\end{theorem}
Remark that Theorem \ref{thm strong} gives a sufficient condition to guarantee the main hypothesis of the  previous theorem, i.e. an IFS being strongly robustly minimal. 
Therefore, Theorems \ref{thm strong} and \ref{thm C} provide a general framework 
to prove robust transitivity, reformulating to the role of blenders.

\def\cprime{$'$}

\end{document}